\documentclass[12pt]{article}

\usepackage{amsmath}
\usepackage{amssymb,latexsym}
\usepackage{amsthm}
\usepackage{bm}
\usepackage{bbm}
\usepackage[all,cmtip]{xy}
\input{diagxy}
\usepackage{url}
\usepackage[colorlinks=true,linkcolor=blue,anchorcolor=blue,citecolor=blue,
     		filecolor=blue,urlcolor=blue]{hyperref}

\newcommand{\C}{\ensuremath{\mathbb{C}}}
\newcommand{\A}{\ensuremath{\mathbb{A}}}

\newcommand{\D}{\ensuremath{\mathbb{D}}}

\newcommand{\op}[1]{\ensuremath{{#1}^{\mathrm{op}}}}

\newcommand{\psh}[1]{\ensuremath{[\op{#1},\mathsf{Set}]}}
\newcommand{\Set}{\ensuremath{\mathsf{Set}}}
\newcommand{\set}{\ensuremath{\mathsf{set}}}
\newcommand{\Cat}{\ensuremath{\mathsf{Cat}}}

\newcommand{\xypbcorner}[1][dr]{\save*!/#1-1.5pc/#1:(-1,1)@^{|-}\restore}
\newcommand{\y}{\ensuremath{\mathsf{y}}} 
\newcommand{\yon}{\ensuremath{\mathsf{y}}} 

\newcommand{\hook}{\ensuremath{\hookrightarrow}}
\newcommand{\mono}{\ensuremath{\rightarrowtail}}
\newcommand{\ra}{\ensuremath{\rightarrow}}
\newcommand{\cof}{\ensuremath{\rightarrowtail}}

\renewcommand{\to}{\ensuremath{\rightarrow}}
\newcommand{\too}{\ensuremath{\longrightarrow}}

\newcommand{\U}{\ensuremath{\mathcal{U}}}
\newcommand{\UU}{\ensuremath{\,\dot{\mathcal{U}}}}

\newcommand{\SSet}{\ensuremath{\,\dot{\Set}}}
\newcommand{\sset}{\ensuremath{\,\dot{\set}}}

\newcommand{\V}{\ensuremath{\mathcal{V}}}
\newcommand{\VV}{\ensuremath{\dot{\mathcal{V}}}}

\newcommand{\Fib}{\ensuremath{\mathsf{Fib}}}

\newcommand{\elem}[1]{\textstyle\int\!{#1}}

\usepackage{tikz}
\usepackage{pdfpages}
\usepackage{tikz-cd}
\newcommand{\pbmark}{\ar[dr, phantom, "\lrcorner" very near start, shift right=.5ex]}	

\newtheorem{theorem}{Theorem}
\newtheorem*{theorem*}{Theorem}
\newtheorem{proposition}[theorem]{Proposition} 
\newtheorem{lemma}[theorem]{Lemma}

\theoremstyle{remark}
\newtheorem{remark}[theorem]{Remark} 
\newtheorem*{remarks*}{Remarks}

\theoremstyle{definition}

\begin{document}

\title{On Hofmann-Streicher universes}
\author{Steve Awodey}
\maketitle

\centerline{\emph{to the memory of Erik Palmgren}}

\medskip

\begin{abstract}
We have another look at the construction by Hofmann and Streicher of a universe $(U,{\mathsf{E}l})$ for the interpretation of Martin-L\"of type theory in a presheaf category $\psh{\C}$.  It turns out that $(U,{\mathsf{E}l})$ can be described as the \emph{categorical nerve} of the classifier $\dot{\Set}^{\mathsf{op}} \to \op{\Set}$ for discrete fibrations in $\Cat$, where the nerve functor is right adjoint to the so-called ``Grothendieck construction'' taking a presheaf $P : \op{\C}\to\Set$ to its category of elements $\int_\C P$.  We also consider change of base for such universes, as well as universes of structured families, such as fibrations.
\end{abstract}



\noindent Let $\widehat{\C} = \psh{\C}$ be the category of presheaves on the small category~$\C$.

\subsection*{1. The Hofmann-Streicher universe}\label{sec:U}

In \cite{HS:1997} the authors define a (type-theoretic) \emph{universe}  $(U, {\mathsf{E}l})$ 
with $U\in\widehat{\C}$ and $\textstyle{\mathsf{E}l} \in \widehat{\int_\C U}$ as follows. For $I\in\C$, set
 \begin{align}
	U(I)\ &=\ \Cat\big(\op{\C/_I}, \Set\big) \label{eq:universeob}\\ 
 	{\mathsf{E}l}(I, A)\ &=\ A(id_I) \label{eq:universeel}
 \end{align}
with an evident associated action on morphisms.  A few comments are required: 
\begin{enumerate}
\item  In \eqref{eq:universeob}, we have taken the underlying set of objects of the category $\widehat{\C/_I}=\psh{\C/_I}$ (in contrast to the specification in \cite{HS:1997}).
\item In \eqref{eq:universeel}, and throughout, the authors steadfastly adopt a ``categories with families'' point of view in describing a morphism 
\begin{equation}\label{eq:universe}
E \too U
\end{equation}
in $\widehat{\C}$ equivalently as an object in
\begin{equation}\label{eq:elements}\textstyle
 \widehat{\C}/_U\ \simeq\ \widehat{\int_{\C}U}\,,
\end{equation}
that is, as a presheaf on the \emph{category of elements} $\int_{\C}U$, where:
\[
E(I)\ =\ {\textstyle \coprod_{A\in U(I)}{\mathsf{E}l}(I, A)}\,.
\]
Thus the argument $(I, A) \in \int_{\C}U$ in \eqref{eq:universeel} consists of an object $I\in\C$ and an element $A\in U(I)$.
\item In order to account for size issues, the authors assume a Grothendieck universe $\U$ in $\Set$, the elements of which are called \emph{small}.  The category $\C$ is then assumed to be small, as are the values of the presheaves, unless otherwise stated.  
\end{enumerate}

The presheaf $U$, which is not small, is regarded as the Grothendieck universe $\U$ ``lifted'' from $\Set$ to $\psh{\C}$.
We will analyse the construction of $(U, {\mathsf{E}l})$  from a slightly different perspective in order to arrive at its basic property as a classifier for small families in $\widehat\C$. 

\subsection*{2. An unused adjunction}

For any presheaf $X$ on $\C$, recall that the category of elements is the comma category,
\[\textstyle
\int_\C X\ =\ \yon_\C/_X\,,
\] 
where $\yon_\C : \C \to \psh\C$ is the Yoneda embedding, which we may supress and write simply $\C/_X$. 
While the category of elements $\int_\C X$ is used in the specification of the Hofmann-Streicher universe $(U, {\mathsf{E}l})$ at the point \eqref{eq:elements}, the authors seem to have missed a trick  which would have simplified things:

\begin{proposition}[\cite{G:1983},\S{28}]
The category of elements functor $\int_\C : \widehat\C \too \Cat$ has a right adjoint, which we denote
\[
\nu_\C : \Cat \too \widehat\C\,.
\]
For a small category $\A$, we call the presheaf $\nu_\C(\A)$ the \emph{$\C$-nerve} of $\A$.
\end{proposition}
\begin{proof}
As suggested by the name, the adjunction $\int_\C\! \dashv \nu_\C$ can be seen as the familiar ``realization $\dashv$ nerve'' construction with respect to the covariant functor $\C/- : \C\to\Cat$, as indicated below.
\begin{equation}\label{eq:nerve}\textstyle
\begin{tikzcd}
	 \widehat\C \ar[rr, swap,"\int_\C"] &&  \ar[ll, swap,bend right=20, "{ \nu_\C}"] \Cat\\  
	 \\
	\C \ar[uu, hook, "\yon"] \ar[rruu, swap,"{\C/_{-}}"] &&
 \end{tikzcd}
 \end{equation}
In detail, for  $\A\in\Cat$ and $c\in\C$, let $\nu_{\C}(\A)(c)$ be the Hom-set of functors,
\begin{align*}
\nu_\C(\A)(c) &= \Cat\big( {\C/_c}\,,\, \A \big)\,,
\end{align*}
with contravariant action on $h : d\to c$ given by pre-composing a functor $P : {\C/_c}\to\A$  with the post-composition functor
\[
{\C/_h} : {\C/_d}\too {\C/_c} \,.
\]
For the adjunction, observe that the slice category $\C/_c$ is the category of elements of the representable functor $\y{c}$\,,
\[\textstyle
\int_\C\y{c}\ \cong\ \C/_c\,.
\]
 Thus for representables $\y{c}$\,, we have the required natural isomorphism
 \[\textstyle
 \widehat\C\big( \y{c}\,,\, \nu_\C(\A) \big)\ \cong\ \nu_\C(\A)(c)\  =\ \Cat\big( {\C/_c}\,,\, \A \big)\ \cong\ \Cat\big( \int_\C\y{c}\,,\, \A \big)\,.
  \]
For arbitrary presheaves $X$, one uses the presentation of $X$ as a colimit of representables over the index category $\int_\C X$, and the easy to prove fact that $\int_\C$ itself preserves colimits.  Indeed, for any category $\D$, we have an isomorphism in $\Cat$,
\[
\varinjlim_{d\in\D}\,\D/_d \ \cong\ \D\,.
\]
\end{proof}

When $\C$ is fixed, as here, we may omit the subscript from the notation $\yon_\C$ and  $\int_\C$ and $\nu_\C$.  The unit and counit maps of the adjunction $\int \dashv \nu$,
\begin{align*}\textstyle
\eta :&\ \textstyle  X \too \nu{\elem{X}}\,, \\
\epsilon :&\ \textstyle  \elem\nu\A \too \A\,,
\end{align*}
 are as follows.  At $c\in\C$, for $x : \y{c}\ra X$, the functor $(\eta_X)_c(x) : \C/_c \to \C/_X$ is just composition with $x$, 
\begin{equation}\label{eq:eta}
(\eta_X)_c(x) = \C/_x : \C/_c \too \C/_X\,.
\end{equation}
For $\A\in\Cat$, the functor $ \epsilon : \int\nu\A \to \A$ takes a pair $(c\in\C, f : \C/_c \to \A)$ to the object $f(1_c) \in \A$,
\[
\epsilon(c,f) = f(1_c).
\]
\begin{lemma}\label{lemma:natpb}
For any $f : Y\to X$, the naturality square below is a pullback.
\begin{equation}\label{eq:naturality}\textstyle
\begin{tikzcd}
	 Y \ar[d, swap,"f"] \ar[r, "{\eta_Y}"] & \nu{\int\!{Y}} \ar[d, "{ \nu{\int\!{f}}}"]\\  
	X \ar[r, swap,"{\eta_X}"] &   \nu{\int\!{X}}.
 \end{tikzcd}
 \end{equation}
\end{lemma}

\begin{proof}
It suffices to prove this for the case $f : X\ra 1$.  Thus consider the square 
\begin{equation}\label{eq:naturalityobject}\textstyle
\begin{tikzcd}
	 X \ar[d] \ar[r, "{\eta_X}"] & \nu{\int\!{X}} \ar[d]\\  
	1\ar[r, swap,"{\eta_1}"] &   \nu{\int\! 1}.
 \end{tikzcd}
 \end{equation}
Evaluating at $c\in\C$ and applying \eqref{eq:eta} then gives the following square in $\Set$.
\begin{equation}\label{eq:naturalityobjecteval}\textstyle
\begin{tikzcd}
	 Xc \ar[d] \ar[r, "{\C/_{-}}"] & \Cat\big( {\C/_c}\,,\, {\C/_X}\, \big) \ar[d]\\  
	1c\ar[r, swap, "{\C/_{-}}"] &   \Cat\big( {\C/_c}\,,\, \C/_1 \big)
 \end{tikzcd}
 \end{equation}
The image of $*\in 1c$ along the bottom is the forgetful functor $U_c : \C/_c\to \C$, and its fiber under the map on the right is therefore the set of functors $F : {\C/_c}\to {\C/_X}$ such that $U_X\circ F = U_c$, where $U_X : \C/_X\to \C$ is also a forgetful functor. But any such $F$ is easily seen to be uniquely of the form $\C/_{x}$ for $x = F(1_c) : \y{c} \to X$.
\end{proof}

\begin{remark} 
Since $\int_\C 1\, \cong\, \C$ the functor  $\int_\C : \widehat\C \to \Cat$ can be factored through the slice category $\Cat/_\C$.
\begin{equation}\label{eq:nerve2}\textstyle
\begin{tikzcd}
	 \widehat\C \ar[rr, "(\int_\C)/_1"] \ar[rrdd, swap, "\int_\C"] 
	 	&& \Cat/_\C \ar[dd, "{\C_!}"] \\  
	 \\
	&& \C 
 \end{tikzcd}
 \end{equation}
 The adjunction $\int_\C\dashv \nu_\C : \Cat \to \widehat\C$ factors as well, but it is the unfactored adjunction that is more useful for the present purpose.
 \end{remark}

\subsection*{3. Classifying families}

For the terminal presheaf $1\in\widehat{\C}$, we have an iso $\elem{1} \cong\C$. So for every $X\in\widehat{\C}$ there is a canonical projection  $\elem X \ra\C$, which is easily seen to be a discrete fibration.  It follows that for any map $Y\to X$ of presheaves, the associated map $\elem Y \to \elem X$ is also a discrete fibration. 
Ignoring size issues for the moment, recall that discrete fibrations in $\Cat$ are classified by the forgetful functor $\op{\dot{\Set}}\to \op{\Set}$ from (the opposites of) the category of pointed sets to that of sets (cf.~\cite{W:2007}).  For every presheaf $X\in\widehat{\C}$, we therefore have a pullback diagram in $\Cat$,
\begin{equation}\label{eq:classifyuniversecat}\textstyle
\begin{tikzcd}
	 \elem X \ar[d] \ar[r] \pbmark & \op{\dot{\Set}} \ar[d]\\  
	\C \ar[r,swap,"X"] &  \op{\Set}.
 \end{tikzcd}
 \end{equation}
 Using $\elem{1} \cong\C$ and transposing by the adjunction $\int \dashv \nu$ then gives a commutative square in $\widehat{\C}$,
\begin{equation}\label{eq:classifyuniversetype}\textstyle
\begin{tikzcd}
	 X \ar[d] \ar[r] & \nu\op{\dot{\Set}} \ar[d]\\  
	1 \ar[r,swap,"\tilde{X}"] &  \nu\op{\Set}.
 \end{tikzcd}
 \end{equation}

\begin{lemma}
The square \eqref{eq:classifyuniversetype} is a pullback in $\widehat{\C}$. More generally, for any map $Y\ra X$ in $\widehat{\C}$, there is a pullback square 
\begin{equation}\label{eq:classifyuniversefamily}\textstyle
\begin{tikzcd}
	 Y \ar[d] \pbmark \ar[r] & \nu\op{\dot{\Set}} \ar[d] \\  
	X \ar[r] &  \nu\op{\Set}\,.
 \end{tikzcd}
 \end{equation}
\end{lemma}

\begin{proof}
Apply the right adjoint $\nu$ to the pullback square \eqref{eq:classifyuniversecat} and paste the naturality square \eqref{eq:naturality} from Lemma \ref{lemma:natpb} on the left, to obtain the transposed square \eqref{eq:classifyuniversefamily} as a pasting of two pullbacks.
\end{proof}

Let us write $\VV \to \V$ for the vertical map on the right in \eqref{eq:classifyuniversefamily};  that is, let
\begin{align}\label{eq:universedef}\textstyle
\VV\, &=\, \nu\op{\dot{\Set}}\\  
\V\, &=\, \nu\op{\Set}.\notag
 \end{align}
 
 We can then summarize our results so far as follows.

 \begin{proposition}\label{prop:Vclassifies}
The nerve $\VV\to\V$  of the classifier for discrete fibrations $\op\SSet\to\op\Set$, as defined in \eqref{eq:universedef}, classifies natural transformations $Y\to X$ in $\widehat{\C}$, in the sense that there is always a pullback square,
\begin{equation}\label{eq:classifyuniversefamily2}\textstyle
\begin{tikzcd}
	 Y \ar[d] \pbmark \ar[r] & \VV \ar[d] \\  
	X \ar[r,swap, "\tilde{Y} "] &  \V.
 \end{tikzcd}
 \end{equation}
The classifying map $\tilde{Y} : X\to \V$ is determined by the adjunction $\int \dashv \nu$ as the transpose of the classifying map of the discrete fibration $\elem Y\to\elem X$.  
\end{proposition}

The classifying map $\tilde{Y} : X\to \V$ of a given a natural transformation $Y\to X$ is, of course, not in general unique. Nonetheless, we can make use of the construction of $\VV\to\V$ as the nerve of the discrete fibration classifier $\op\SSet\to\op\Set$, for which classifying functors $\C \to \op\Set$ are unique up to natural isomorphism, to infer the following proposition, which plays a role in \cite{Shu:15,GSS:22} and elsewhere.

\begin{proposition}[Realignment]\label{prop:realignment}
Given a monomorphism $c : C\cof X$ and a family $Y\to X$, let $y_c : C \to \V$ classify the pullback $c^*Y\to C$.  Then there is a classifying map $y: X \to \V$ for $Y\to X$ with $y\circ c = y_c$.
\begin{equation}\label{diagram:realignment}
\begin{tikzcd}
c^*Y \ar[dd] \ar[rd] \ar[rr] && \VV \ar[dd] \\
& Y \ar[dd] \ar[ru, dotted] & \\
C  \ar[rd, tail,swap, "c"] \ar[rr, near start, "y_c"] && \V  \\
& X \ar[ru, dotted, swap, "y"] &
\end{tikzcd}
\end{equation}
\end{proposition}
\begin{proof}
Transposing the realignment problem \eqref{diagram:realignment} for presheaves across the adjunction $\int\dashv \nu$ results in the following realignment problem for discrete fibrations.
\begin{equation}\label{diagram:realignment2}
\begin{tikzcd}
\elem  c^*Y \ar[dd] \ar[rd] \ar[rr] && \op{\dot{\Set}}  \ar[dd] \\
&\elem  Y \ar[dd] \ar[ru, dotted] & \\
\elem  C  \ar[rd, tail,swap, "{\elem  c}"] \ar[rr, near start, "\widetilde{y_c}"] && \op{{\Set}}   \\
& \elem  X \ar[ru, dotted, swap, "\tilde{y}"] &
\end{tikzcd}
\end{equation}
The category of elements functor $\int $ is easily seen to preserve pullbacks, hence monos; thus let us consider the general case of a functor  $C : \C \mono \D$ which is monic in $\Cat$, a pullback of discrete fibrations as on the left below, and a presheaf $E : \C \to  \op{{\Set}}$ with $\elem E \cong \mathbb{E}$ over $\C$. 
\begin{equation}\label{diagram:realignment3}
\begin{tikzcd}
\mathbb{E} \ar[dd] \ar[rd] \ar[rr] && \op{\dot{\Set}}  \ar[dd] \\
& \mathbb{F}  \ar[dd] \ar[ru, dotted] & \\
\C  \ar[rd, tail,swap, "{C}"] \ar[rr, near start, "E"] && \op{{\Set}}   \\
& \D \ar[ru, dotted, swap, "F"] &
\end{tikzcd}
\end{equation}
We seek $F : \D \to  \op{{\Set}}$ with $\elem F \cong \mathbb{F}$ over $\D$ and $F\circ C = E$.  Let $F_0 : \D \to  \op{\Set}$ with $\elem F_0 \cong \mathbb{F}$ over $\D$.  Since $F_0\circ C$ and $E$ both classify $\mathbb{E}$, there is a natural iso $e : F_0\circ C \cong E$.
Consider the following diagram
\begin{equation}\label{diagram:realignment4}
\begin{tikzcd}
\C \ar[dd, tail, swap, "{C}"]  \ar[rr, "e"] && {\op{(\Set^{\cong})}}  \ar[dd,"{p_1}"] \ar[r,swap,"{p_2}"] &  {\op{\Set}} \\
&&&\\
\D  \ar[rr,swap, "F_0"]  \ar[rruu, dotted, swap, "f"] && \op{\Set} & 
\end{tikzcd}
\end{equation}
where $\Set^{\cong}$ is the category of isos in $\Set$, with $p_1, p_2$ the (opposites of the) domain and codomain projections.  There is a well-known weak factorization system on $\Cat$ (part of the ``canonical model structure'') with injective-on-objects functors on the left and isofibrations on the right.  Thus there is a diagonal filler $f$ as indicated.  The functor $F := p_2 f : \D \to \op{\Set}$ is then  the one we seek.
\end{proof}

Of course, as defined in \eqref{eq:universedef}, the classifier $\VV\to\V$ cannot be a map in $\widehat{\C}$, for reasons of size; we now address this.

\subsection*{4. Small maps}

Let $\alpha$ be a cardinal number, and call the sets that are strictly smaller than it $\alpha$-\emph{small}.  Let $\Set_\alpha\hook\Set$ be the full subcategory of $\alpha$-small sets.  
Call a presheaf $X : \op{\C} \to \Set$ $\alpha$-small if all of its values are $\alpha$-small sets, and thus if, and only if, it factors through $\Set_\alpha\hook\Set$. Call a map $f:Y\to X$ of presheaves $\alpha$-small if all of the fibers $f_c^{-1}\{ x\} \subseteq Yc$ are $\alpha$-small sets (for all $c\in\C$ and $x\in Xc$). The latter condition is of course equivalent to saying that, in the pullback square over the element $x:\y{c} \to X$, 
\begin{equation}\label{eq:smallmap}\textstyle
\begin{tikzcd}
	 Y_x \ar[d] \pbmark \ar[r] & Y \ar[d, "f"] \\  
	\y{c} \ar[r,swap,"x"] &  X,
 \end{tikzcd}
 \end{equation}
the presheaf $Y_x$ is $\alpha$-small.


Now let us restrict the specification \eqref{eq:universedef} of $\VV\to\V$ to the $\alpha$-small sets:
\begin{align}\label{eq:universedefalpha}\textstyle
\VV_\alpha\, &=\, \nu \dot{\Set^{\mathsf{op}}_\alpha}\\  
\V_\alpha\, &=\, \nu \Set^{\mathsf{op}}_\alpha. \notag
 \end{align}
Then the evident forgetful map $\VV_\alpha\to\V_\alpha$ \emph{is} a map in the category $\widehat{\C}$ of presheaves, and it is in fact $\alpha$-small. Moreover, it has the following basic property, which is just a restriction of the basic property of $\VV\to\V$ stated in Proposition \ref{prop:Vclassifies}.

 \begin{proposition}\label{prop:Valphaclassifies}
The map $\VV_\alpha\to\V_\alpha$ classifies $\alpha$-small maps $f:Y\to X$ in $\widehat{\C}$, in the sense that there is always a pullback square,
\begin{equation}\label{eq:classifyuniversefamilyalpha}\textstyle
\begin{tikzcd}
	 Y \ar[d] \pbmark \ar[r] & \VV_\alpha \ar[d] \\  
	X \ar[r,swap, "\tilde{Y}"] &  \V_\alpha.
 \end{tikzcd}
 \end{equation}
The classifying map $\tilde{Y} : X\to \V_\alpha$ is determined by the adjunction $\int \dashv \nu$ as (the factorization of) the transpose of the classifiyng map of the discrete fibration $\elem X\to\elem Y$. 
\end{proposition}

\begin{proof} If $Y\to X$ is $\alpha$-small, its classifying map $\tilde{Y} : X\to\V$ factors through $\V_\alpha \hook \V$, as indicated below, 
\begin{equation}\label{eq:classifyuniversetype2}\textstyle
\begin{tikzcd}
	 Y \ar[d] \ar[rr, bend left] \ar[r] & \nu\op{\dot{\Set_\alpha}} \ar[d] \ar[r,hook] & \nu\op{\dot{\Set}} \ar[d]\\  
	X \ar[rr, bend right, swap,"\tilde{Y}"] \ar[r] &  \nu\op{\Set_\alpha} \ar[r,hook] &  \nu\op{\Set},
 \end{tikzcd}
 \end{equation}
in virtue of the following adjoint transposition,
\begin{equation}\label{eq:adjointtranspose}\textstyle
\begin{tikzcd}
	 \elem Y \ar[d] \ar[rr, bend left] \ar[r] & \op{\dot{\Set_\alpha}} \ar[d] \ar[r,hook] & \op{\dot{\Set}} \ar[d]\\  
	 \elem X \ar[rr, bend right, ] \ar[r]  &  \op{\Set_\alpha} \ar[r,hook]  &  \op{\Set}.
 \end{tikzcd}
  \end{equation}
Note that the square on the right is evidently a pullback, and the one on the left therefore is, too, because the outer rectangle is the classifying pulback of the discrete fibration $\elem Y \to \elem X$, as stated.  Thus the left square in \eqref{eq:classifyuniversetype2} is a pullback.
\end{proof}

\subsection*{5. Examples}

\begin{enumerate}
\item Let $\alpha = \kappa$ a strongly inaccessible cadinal, so that $\mathsf{ob}({\Set_\kappa})$ is a Grothendieck universe.  Then the Hofmann-Streicher universe of \eqref{eq:universe} is recovered in the present setting as the $\kappa$-small map classifier
\begin{equation*}
E\, \cong\, \VV_\kappa \too \V_\kappa\, \cong\, U
\end{equation*}
 in the sense of Proposition \ref{prop:Valphaclassifies}.  Indeed, for $c\in\C$, we have 
 \begin{align}
  \V_{\kappa}{c}\ &=\ \nu(\Set^{\mathsf{op}}_\kappa)(c) = \Cat\big( {\C/_c}\,,\, \Set^{\mathsf{op}}_\kappa \big)\  =\ \mathsf{ob}(\widehat{\C/_c})\ =\ U{c} \,.
   \end{align} 
For $\VV_{\kappa}$ we then have,
   \begin{align}\label{eq:veedotc}
   \VV_{\kappa}{c}\ =\ \nu(\SSet^{\mathsf{op}}_\kappa)(c)\ &=\ \Cat\big( {\C/_c}\,,\, \SSet^{\mathsf{op}}_\kappa \big) \notag \\ 
   &\cong\ {\textstyle \coprod_{A\in\V_{\kappa}{c}}\Cat_{{\C/_c}}\big( {\C/_c}\,,\, A^*\Set^{\mathsf{op}}_\kappa \big)}
   \end{align}
   where the $A$-summand in \eqref{eq:veedotc} is defined by taking sections of the  pullback indicated below.
   \begin{equation}\label{eq:pbforindexing}\textstyle
\begin{tikzcd}
	A^*\Set^{\mathsf{op}}_\kappa \ar[d] \ar[r] \pbmark & \SSet^{\mathsf{op}}_\kappa \ar[d]\\  
	\C/_c \ar[r,swap,"A"] \ar[u, bend left, dotted] \ar[ur, dotted] &  \Set^{\mathsf{op}}_\kappa
 \end{tikzcd}
 \end{equation}
 But $A^*\Set^{\mathsf{op}}_\kappa\ \cong\ {\textstyle \int_{\C/_c}\!A}$ over $\C/_c\,$, and sections of this discrete fibration in $\Cat$ correspond uniquely to natural maps $1\to A$ in $\widehat{{\C/_c}}$.  Since $1$  is representable in $\widehat{{\C/_c}}$ we can continue \eqref{eq:veedotc} by
  \begin{align*}
   \VV_{\kappa}{c}\ &\cong\ {\textstyle \coprod_{A\in \V_{\kappa}{c}}\Cat_{{\C/_c}}\big( {\C/_c}\,,\, A^*\Set^{\mathsf{op}}_\kappa \big)}\\
   	&\cong\ {\textstyle \coprod_{A\in \V_{\kappa}{c}} \widehat{{\C/_c}}(1, A)}\\
	&\cong\ {\textstyle \coprod_{A\in \V_{\kappa}{c}} A(1_c) } \\
	& =\ {\textstyle \coprod_{A\in \V_{\kappa}{c}} {\mathsf{E}l}(\langle c, A\rangle)}\\
	& =\  E c\,.
   \end{align*}
 
\item By functoriality of the nerve $\nu : \Cat \to \widehat{\C}$, a sequence of Grothendieck universes $$\U \subseteq {\U\,}' \subseteq ...$$ in $\Set$ gives rise to a (cumulative) sequence of type-theoretic universes $$\V \mono {\V\,}' \mono ...$$ in $\widehat{\C}$. More precisely, there is a sequence of  cartesian squares,
\begin{equation}\label{eq:Vhierarchy}\textstyle
\begin{tikzcd}
	 \VV \ar[d] \ar[r,tail] \pbmark & {\VV\,'} \ar[d] \ar[r,tail] \pbmark & \dots \\  
	 \V  \ar[r, tail]  &  {\V\,}' \ar[r,tail]  & \dots\,,
 \end{tikzcd}
  \end{equation}
in the image of $\nu : \Cat\too\widehat\C$, classifying small maps in $\widehat\C$ of increasing size, in the sense of Proposition \ref{prop:Valphaclassifies}.

\item Let $\alpha = 2$ so that $1\to 2$ is the subobject classifier of $\Set$, and 
$$\mathbbm{1} = \SSet^{\mathsf{op}}_2 \too  \Set^{\mathsf{op}}_2 = \mathbbm{2}$$ is then a classifier in $\Cat$ for \emph{sieves}, i.e.\ full subcategories $\mathbb{S}\hook\A$ closed under the domains of arrows $a\to s$ for $s\in\mathbb{S}$.  The nerve $\VV_{2}  \to \V_{2}$ is then exactly the subobject classifier $1\to\Omega$ of $\widehat\C$,
\[
1 = \nu \mathbbm{1} = \VV_{2} \too  \V_{2} = \nu \mathbbm{2} = \Omega  \,.
\]

\item Let $i : \mathbbm{2} \hook \Set_\kappa$ and $ p : \Set_\kappa \to \mathbbm{2} $ be the embedding-retraction pair with $i : \mathbbm{2} \hook \Set_\kappa$ the inclusion of the full subcategory on the sets $ \{0, 1\}$ and $p : \Set_\kappa \to \mathbbm{2}$ the retraction that takes $0 = \emptyset$ to itself, and everything else (i.e.\ the non-empty sets) to $1 = \{\emptyset\}$. There is a retraction (of arrows) in $\Cat$,
\begin{equation}\label{eq:Setretraction}\textstyle
\begin{tikzcd}
	 \mathbbm{1} \ar[d]  \ar[r,hook] \pbmark & \SSet_\kappa \ar[d] \ar[r] &\mathbbm{1} \ar[d] \\  
	 \mathbbm{2}   \ar[r,hook,swap, "i"]  &  \Set_\kappa \ar[r, swap, two heads, "p"]  & \mathbbm{2} 
 \end{tikzcd}
  \end{equation}
  where the left square is a pullback.  
    
By the functoriality of (\,$\op{-}$ and) $\nu : \Cat \to \widehat{\C}$ we then have a retract diagram in $\widehat\C$, again with a pullback on the left, 
\begin{equation}\label{eq:Vretraction}\textstyle
\begin{tikzcd}
	 1 \ar[d]  \ar[r,hook] \pbmark & \VV_\kappa \ar[d] \ar[r] &1\ar[d] \\  
	 \Omega   \ar[r,hook,swap, "\{-\}"]  &  \V_\kappa  \ar[r, swap, two heads, "{[-]}"]  &  \Omega 
 \end{tikzcd}
  \end{equation}
  where for any $\phi : X\to \Omega$ the subobject $\{\phi\}\mono X$ is classified as a small map by the composite $\{\phi\} : X\to \V_\kappa$\,, and for any small map $A\to X$, the image $[A] \mono X$ is classified as a subobject by the composite $[\alpha] : X\to \V_\kappa \to \Omega$\,, where $\alpha : X\to \V_\kappa$ classifies $A\to X$.  The idempotent composite $$|\!|\!-\!|\!| = \{[-]\} : \V_\kappa  \too \V_\kappa $$ is the \emph{propositional truncation modality} in the natural model of type theory given by $\VV_\kappa\to \V_\kappa$ (see \cite{AGH}).
\end{enumerate}

\subsection*{6. Change of base}

Let  $F : \C \to \D$ in $\Cat$ and consider the base change:

\begin{equation}\label{diagram:basechange1}
\begin{tikzcd}
\C \ar[dd, swap, "{F}"]  \ar[rr, hook,"\yon_\C"] && \Set^{\op{\C}}  \ar[dd,swap,shift right=3,"{F_!}"]  \ar[dd,shift left=3,"{F_*}"]\\
&&\\
\D  \ar[rr, hook,swap, "\yon_\D"] && \ar[uu,"{F^*}" description] \Set^{\op{\D}}  
\end{tikzcd}
\end{equation}
How is the universal (small) map $\VV_\C\to \V_\C$ in $\widehat{\C}$ related to $\VV_\D\to \V_\D$ in~$\widehat{\D}$\,?

For each $c\in\C$ we have the sliced functor $F/_{c} : \C/_c \to \D/_{Fc}$ which is the component at $c\in\C$ of a natural transformation $F/ : \C/ \to \D/\circ F$ as functors $\C \to \Cat$.
\begin{equation}\label{diagram:basechange2}
\begin{tikzcd}
\C \ar[dd,swap,"F"] \ar[dd,"{\qquad{\Downarrow\ F/}}"] \ar[rrrd,"{\C/}"] && \\
&&& \Cat \\
\D  \ar[rrru,swap, "{\D/}"] &&&
\end{tikzcd}
\end{equation}
Indeed, for each $h : c \to c'$ there is a commutative square:
\begin{equation}\label{diagram:basechange3}
\begin{tikzcd}
\C/_c \ar[dd, swap, "{F/_{c}}"]  \ar[rr, "{h\circ}"] && \C/_{c'} \ar[dd, "{F/_{c'}}"]  \\
&&\\
\D/_{Fc}  \ar[rr,swap, "{(Fh)\circ}"] && \D/_{Fc'} 
\end{tikzcd}
\end{equation}

The 2-cell $F/ : \C/ \to \D/\circ F$ in \eqref{diagram:basechange2} (left Kan) extends to one,
$$\textstyle \int_F : \int_\C \Longrightarrow \int_\D \circ\,F_!$$
\begin{equation}\label{diagram:basechange4}
\begin{tikzcd}
\widehat\C \ar[dd,swap,"F_!"] \ar[dd,"{\qquad{\Downarrow\ \int_F}}"] \ar[rrrd,"{\int_\C}"] &&& \\
&&& \Cat\,, \\
\widehat\D \ar[rrru,swap, "{\int_\D}"] &&&
\end{tikzcd}
\end{equation}
which has the following 2-categorical mate.
$$\textstyle \nu_F : F^*\circ \nu_\D \Longrightarrow \nu_\C$$
\begin{equation}\label{diagram:basechange5}
\begin{tikzcd}
\widehat\C  &&& \\
&&& \ar[lllu,swap,"{\nu_\C}"] \ar[llld,"{\nu_\D}"] \Cat \\
\widehat\D  \ar[uu,"F^*"] \ar[uu,swap,"{\qquad{\Uparrow\ \nu_F}}"] &&&
\end{tikzcd}
\end{equation}
Evaluating at the (small) discrete fibration classifier $\op\sset\to \op\set$, we obtain a commutative square in $\widehat\C$ of the form
\begin{equation}\label{diagram:basechange6}
\begin{tikzcd}
F^*\VV_D \ar[d]  \ar[r, "{\dot{\gamma}}"] & \VV_\C \ar[d]  \\
F^*\V_\D \ar[r,swap, "{\gamma}"] & \V_\C
\end{tikzcd}
\end{equation}
where we have written 
\begin{align*}
\dot{\gamma}\, &=\, (\nu_F)_{\op{\dot{\set}}}\\
\gamma\, &=\, (\nu_F)_{\op{\set}}
\end{align*}
for the components of $\nu_F$.

\begin{proposition}[Base change for universes]\label{prop:universebasechange}
For any functor $F: \C\to \D$, the comparison square \eqref{diagram:basechange6} for universes is a pullback in $\widehat\C$.
\end{proposition}
For the proof, we require the following.
\begin{lemma}\label{lemma:locallyfinal}
For any functor $F: \C\to \D$ and any $c\in\C$, the sliced functor $F/_c: \C/_c \too \D/_{Fc}$ is final.
\end{lemma}
\begin{proof}
Recall that a functor $F: \C\to \D$ is by definition \emph{final} if for all $G: \D\to\Set$, the canonical map $\varinjlim G\circ F \to \varinjlim G$ is an iso, and this holds just if, for all $d\in\D$, the comma category $d/_F$ is connected.  Moreover, recall from \cite{Street-Walters-1973} that the \emph{comprehensive} factorization system $(\mathsf{Fin}, \mathsf{dFib})$ on $\Cat$ has the final functors as left orthogonal to the discrete fibrations.  
But now note that the slice categories $\C/_c$ and $\D/_{Fc}$ have terminal objects, preserved by $F/_c: \C/_c \too \D/_{Fc}$.  It therefore suffices to observe that the inclusion functor $i : \{1_\C\} \hook \C$ of a terminal object is always final, so the same is true for $F/_c: \C/_c \too \D/_{Fc}$ by a familiar factorization property of the left maps in an orthogonal factorization system.
\end{proof}

\begin{proof}(of Proposition \ref{prop:universebasechange})
Evaluating \eqref{diagram:basechange6} at $c\in \C$ we obtain the following diagram of sets and functions.
\begin{equation}\label{diagram:basechange7}
\begin{tikzcd}
\Cat\big( {\D/_{Fc}}\,,\, \sset^{\mathsf{op}} \big) \ar[dd]   \ar[r, equals]  & (F^*\VV_D)c \ar[dd]  \ar[r, "{\dot{\gamma}_c}"] 
	& (\VV_\C)c \ar[dd]  \ar[r, equals] & \Cat\big( {\C/_c}\,,\, \sset^{\mathsf{op}} \big) \ar[dd]    \\
&&&\\
 \Cat\big( {\D/_{Fc}}\,,\, \set^{\mathsf{op}} \big)   \ar[r, equals]  & (F^*\V_\D)c \ar[r,swap, "{\gamma_c}"] 
 	& (\V_\C)c \ar[r, equals] & \Cat\big( {\C/_c}\,,\, \set^{\mathsf{op}} \big)  
\end{tikzcd}
\end{equation}
The outer square of this diagram is a pullback exactly if every square as follows has a diagonal filler.
\begin{equation}\label{diagram:basechange8}
\begin{tikzcd}
\C/_c \ar[d,swap, "F/_c"]  \ar[r] & \sset^{\mathsf{op}}  \ar[d]  \\
\D/_{Fc} \ar[r] \ar[ru, dotted] & \set^{\mathsf{op}}
\end{tikzcd}
\end{equation}
Since the map on the right is the universal small discrete fibration, this condition obtains just in case the map on the left is left orthogonal with respect to all small discrete fibrations (for ``small'' sufficiently large), which holds just if it is a final functor, by the comprehensive factorization system \cite{Street-Walters-1973}. But by Lemma \ref{lemma:locallyfinal}, this is the case for every $c\in\C$.
\end{proof}

\subsection*{7. Classifying fibrations}

Given the universal small family $\VV\to \V$ of Proposition \ref{prop:Valphaclassifies}, we can construct a classifier $\UU\to\U$ for any \emph{structured notion of fibration} using the method of \emph{classifying types} developed in \cite{AwodeyCCMC}.  We then consider the behavior of such universal maps under base change.

Suppose that for each object $X$ and family $A\to X$, there is an object $\Fib(A) \ra X$ over $X$ that classifies fibration structures on $A\to X$, in the sense that there is a bijection, natural in $X$, between sections of $\Fib(A) \ra X$ and (some notion of) fibration structures on $A\to X$.
Naturality in $X$ means that $\Fib(A) \ra X$ is stable under pullback, in the sense that for any $f:Y\ra X$, we have two pullback squares:
\begin{equation}\label{diagram:fibstable}
\xymatrix{
f^*A \ar[d] \ar[r]  & A \ar[d]\\
Y \ar[r]_{f} &X\\
\Fib(f^*A) \ar[u] \ar[r] & \Fib(A) \ar[u].
}
\end{equation}
Thus,
\[
\Fib(f^*A) \cong f^* \Fib(A)\,.
\]

It then follows from Proposition \ref{prop:Valphaclassifies} that, if $A\ra X$ is small, then $\Fib(A) \ra X$ is itself a pullback of the analogous object $\Fib(\VV) \ra \V$ constructed from the universal small family $\VV\ra\V$. So again there are two pullback squares,
\begin{equation}\label{diagram:tfib1}
\xymatrix{
A \ar[d] \ar[r]  & \VV \ar[d]\\
X \ar[r] & \V\\
\Fib(A)\ar[u] \ar[r] & \Fib(\VV)\ar[u]\,
}
\end{equation}
where the indicated map $X \to \V$ is the canonical classifier of $A\to X$.

\begin{proposition}\label{prop:classTFib}
There is a \emph{universal small fibration}, 
\[
\UU\too\U.
\]
 Every small fibration $A \ra X$ is a pullback of $\UU\ra\U$ along a canonical classifying map $X\ra \U$.
\begin{equation}\label{diagram:classifytf}
\xymatrix{
A \ar[d] \ar[r]  \xypbcorner & \UU\ar[d]\\
X \ar[r] & \U
}
\end{equation}
\end{proposition}
\begin{proof}
We can take 
\[
\U := \Fib(\VV),
\]
which comes with its canonical projection $\Fib(\VV) \ra \V$ as in diagram \eqref{diagram:tfib1}.  Now define $\UU\ra\U$ by pulling back the universal small family,
\[
\xymatrix{
\UU \ar[d] \ar[r] \xypbcorner  & \VV \ar[d] \\
\U \ar[r] & \V.
}
\]
Consider the following diagram, in which all the squares (including the distorted ones) are pullbacks, with the outer vertical one coming from Proposition \ref{prop:Valphaclassifies} and the lower horizontal one from \eqref{diagram:tfib1}.
\begin{equation}\label{diagram:classifytf2}
\xymatrix{
&A \ar[ddd] \ar[rrr]   \ar@{.>}[rrd] &&& \VV\ar[ddd]\\
& && \UU \ar[d] \ar[ru]  &\\
\Fib(A) \ar[rd] \ar[rrr] |<<<<<<<<<<<<\hole  &&& \U \ar[rd] &\\
&X \ar[rrr]_{a} \ar@{.>}@/^1pc/[lu]^\alpha \ar@{.>}[rru]_{\alpha'} &&& \V.
}
\end{equation}
By assumption, a fibration structure $\alpha$ on $A\ra X$ is a section of $\Fib(A)\to X$, which is the pullback of $\Fib(\VV)=\U$ along the classifying map $a: X \to \V$,
\[
\Fib(A) =  a^*\Fib(\VV) = a^*\U\,.
\]
Such sections therefore correspond uniquely to factorizations $\alpha'$ of $a : X\to\V$, as indicated, which in turn induce pullback squares of the required kind \eqref{diagram:classifytf}.

Finally, observe that the map $\UU\to\U$ has a canonical fibration structure. Indeed, consider the following diagram, in which both squares are pullbacks.
\begin{equation}\label{diagram:fibisfib}
\xymatrix{
\UU \ar[d] \ar[r]  & \VV \ar[d]\\
\U \ar[r] & \V\\
\Fib(\UU) \ar[u] \ar[r] & \Fib(\VV)\ar[u].
}
\end{equation}
Since $\Fib(\VV)$ is the object of fibration structures on $\VV\ra\V$, its pullback $\Fib(\UU)$ is the object of fibration structures on $\UU\ra\U$.
But since $\U = \Fib(\VV)$ by definition, the lower square is the pullback of $\Fib(\VV)\ra \V$ against itself, which does indeed have a distinguished section, namely the diagonal
\[
\Delta : \Fib(\VV) \ra \Fib(\VV)\times_\V\Fib(\VV).
\]
\end{proof}

\subsection*{8. Change of base for universal fibrations}

Now let  $F : \C \to \D$ in $\Cat$ with $\VV_\C\to \V_\C$ in $\widehat{\C}$ and $\VV_\D\to \V_\D$ in~$\widehat{\D}$, related by the base change geometric morphism $$F_!\dashv F^*\dashv F_* : \widehat{\C} \too \widehat{\D}$$ as in Proposition \ref{prop:universebasechange}, and suppose that we have a structured notion of fibration in $\widehat\C$,  classified by $\UU_\C \to \U_\C$ as in Proposition \ref{prop:classTFib}.  Then we can can transfer a structured notion of fibration to $\widehat{\D}$ by declaring $B\to Y$ to be a fibration in $\widehat{\D}$ just if $F^*B\to F^*Y$ is one in $\widehat\C$; more precisely, we define a \emph{fibration structure} on $B\to Y$ to be one on $F^*B\to F^*Y$.  These structures are classified in $\widehat{\D}$ as follows.

Fibration structures on $F^*B\to F^*Y$ are classified in $\widehat{\C}$ by sections $\sigma$ of $\Fib(F^*B)\to F^*Y$ ,
\begin{equation}\label{diagram:changebasefib1}
\xymatrix{
& F^*B \ar[dd] \\
\Fib(F^*B) \ar[rd]  & \\
& F^*Y  \ar@{.>}@/^1pc/[lu]^\sigma .
}
\end{equation}
Applying the right adjoint $F_*$ and pulling back along the unit $\eta$, we obtain an object over $Y$ that classifies such sections.
\begin{equation}\label{diagram:changebasefib1.5}
\Fib(B) := \eta^*F_*\Fib(F^*B) \too Y
\end{equation}

\begin{equation}\label{diagram:changebasefib2}
\xymatrix{
& B \ar[dd] && F_*F^*B \ar[dd] \\
\eta^*F_*\Fib(F^*B) \ar[rr] \ar[rd]   && F_*\Fib(F^*B) \ar[rd]  & \\
& Y  \ar@{.>}@/^1pc/[lu] \ar[rr]_\eta  \ar@{.>}[ru] &&  F_*F^*Y
}
\end{equation}
Indeed, sections of $\Fib(B) \to Y$ correspond bijectively to lifts of the unit $\eta$ across the image $F_*\Fib(F^*B) \to F_*F^*Y $ under $F_*$, which are exactly sections of $\Fib(F^*B) \to F^*Y $.

As before, it suffices to do this construction once in the ``universal case'', %
\begin{equation}\label{diagram:changebasefib3}
\xymatrix{
& \VV_\D \ar[dd] && F_*F^*\VV_\D \ar[dd] \\
\eta^*F_*\Fib(F^*\VV_\D) \ar[rr] \ar[rd]   && F_*\Fib(F^*\VV_\D) \ar[rd]  & \\
& \V_\D   \ar[rr]_\eta  &&  F_*F^*\V_\D
}
\end{equation}
to obtain the classifying type for fibrations in $\widehat\D$ as (the domain of) the indicated map
\begin{equation}\label{diagram:changebasefib4}
\U_\D := \eta^*F_*\Fib(F^*\VV_\D) \too \V_\D\,.
\end{equation}

\begin{proposition}
Let $\UU_\D \to \U_\D$ be the pullback indicated below,
\[
\xymatrix{
\UU_\D \ar[d] \ar[r] \xypbcorner  & \VV_\D \ar[d] \\
\U_\D \ar[r] & \V_\D.
}
\]
where $\U_\D \to  \V_\D$ is as defined in \eqref{diagram:changebasefib4}.
Then $\UU_\D \to \U_\D$ classifies fibrations in~$\widehat\D$. 
\end{proposition}

\begin{proof}
Proposition \ref{prop:classTFib} applies once we know that $\Fib(-)$ as defined in \eqref{diagram:changebasefib1.5} is stable under pullback.  It clearly suffices to show that, for each $B\to Y$, the horizontal square below is a pullback when $b : Y\to \V_\D$ is the canonical classifying map.
\begin{equation}\label{diagram:changebasefib5}
\xymatrix{
& B \ar[dd]  \ar[rr] \xypbcorner && \VV_\D \ar[dd] \\
\Fib(B) \ar[rr] \ar[rd]   && \Fib(\VV_\D) \ar[rd]  & \\
& Y  \ar[rr]_b  &&  \V_\D
}
\end{equation}

\[
\Fib(B) \cong b^*\Fib(\VV_\D)
\]
Inspecting the definitions \eqref{diagram:changebasefib1.5} and \eqref{diagram:changebasefib4}, this follows from the naturality of the units $\eta$,
\begin{equation}\label{diagram:changebasefib5.1}
\xymatrix{
Y  \ar[d]_b \ar[rr]^\eta && F_*F^*Y \ar[d]^{F_*F^*b} \\
\U_\D \ar[rr]_\eta && F_*F^*\V_\D.
}
\end{equation}
Indeed, consider the following cube, in which the front face is the naturality square \eqref{diagram:changebasefib5.1} and the top and bottom faces are the defining pullbacks \eqref{diagram:changebasefib2} and \eqref{diagram:changebasefib3} of $\Fib(B)$ and $\U_\D = \Fib(\VV_\D)$ respectively.
\begin{equation}\label{diagram:changebasefib5.5}
\xymatrix{
\Fib(B) \ar@{=}[r] \ar[ddd] & \eta^*F_*\Fib(F^*B) \ar[ddd] \ar[rr] \ar[rd]   && F_*\Fib(F^*B) \ar[ddd] |<<<<<<<<<<<<\hole  \ar[rd]  & \\
&& Y \ar[ddd]_b  \ar[rr]^<<<<<<<<<<<<<<{\eta} && F_*F^*Y \ar[ddd]^{F_*F^*b} \\
&&&& \\
\Fib(\VV_\D) \ar@{=}[r] & \eta^*F_*\Fib(F^*\VV_\D) \ar[rr] |<<<<<<<<<<<<\hole  \ar[rd]   && F_*\Fib(F^*\VV_\D) \ar[rd]  & \\
&& \V_\D   \ar[rr]_\eta &&  F_*F^*\V_\D
}
\end{equation}
It suffices to show that the right face is a pullback, and since $F_*$ preserves pullbacks it therefore suffices to show that the horizontal square below is a pullback.
\begin{equation}\label{diagram:changebasefib8}
\xymatrix{
&F^*B \ar[dd] \ar[rr]  \xypbcorner && F^*\VV_\D\ar[dd]\\
\Fib(F^*B) \ar[rd] \ar[rr] |<<<<<<<<<<<<<<\hole  && \Fib(F^*\VV_\D) \ar[rd] &\\
&F^*Y \ar[rr]_{F^*b}  && F^*\V_\D.
}
\end{equation}
But that follows from the stability of $\Fib$ in $\widehat\C$ and preservation by $F^*$ of the classifying pullback square for $B\to Y$.
\end{proof}

\subsection*{Acknowledgement}

Thanks to Mathieu Anel, Thierry Coquand, Marcelo Fiore, and Emily Riehl for discussions, and to Evan Cavallo, Ivan Di Liberti, and Taichi Uemura for help with the references. This material is based upon work supported by the Air Force Office of Scientific Research under awards FA9550-21-1-0009 and FA9550-20-1-0305.

\bibliographystyle{alpha}
\bibliography{../references}

\end{document}